\documentclass{amsart}

\usepackage[top=1.2in, bottom=1.2in, left=1.3in, right=1.3in]{geometry}
\usepackage{amsthm} 
\usepackage{enumerate} 
\usepackage[leqno]{amsmath} 
\usepackage{latexsym,amsfonts,amssymb} 
\usepackage[all]{xy} 
\usepackage{verbatim}
\SelectTips{eu}{} 
\SilentMatrices 
\usepackage{hyperref}

\usepackage{amsthm}
\usepackage{enumerate}
\usepackage[leqno]{amsmath}
\usepackage{latexsym,amsfonts,amssymb}
\usepackage[all]{xy} \SelectTips{eu}{} \SilentMatrices
\usepackage{hyperref}

\newcommand{\D}{\mathsf{D}}
\newcommand{\K}{\mathsf{K}}

\newcommand{\C}{\mathsf{C}}

\renewcommand{\S}{S}
\newcommand{\tS}{\mathsf{S}}
\newcommand{\T}{\mathsf{T}}

\newcommand{\Ls}{\mathsf{L}}

\newcommand{\Homs}[3]{\operatorname{Hom}_{#1}^{*}(#2,#3)}
\newcommand{\Homd}[4]{\operatorname{Hom}^{#1}_{#2} (#3, #4)}
\newcommand{\Hom}[3]{\operatorname{Hom}_{#1}(#2,#3)}

\newcommand{\Ext}[4]{\operatorname{Ext}_{#2}^{#1}(#3,#4)}

\newcommand{\brm}[2]{\lbrace \, #1 \, | \, #2 \rbrace}

\newcommand\id{\operatorname{inj\,dim}}

\newcommand{\Thick}[2]{\mathsf{thick}_{#1}(#2)}
\newcommand{\Loc}[2]{\mathsf{loc}_{#1}(#2)}

\newcommand\Img{\operatorname{Im}}

\newcommand{\xra}{\xrightarrow}

\renewcommand{\bf}{\mathbf{f}}

\newcommand{\li}{ < \infty}

\newcommand{\Df}{\mathsf{D}^{\mathsf{f}}(R)}

\theoremstyle{plain}
\newtheorem{theorem}{Theorem}[section]

\newtheorem*{introthm}{Theorem}
\newtheorem{Corollary}{Corollary}

\newtheorem{prop}[theorem]{Proposition}
\newtheorem{cor}[theorem]{Corollary}

\theoremstyle{definition}
\newtheorem{chunk}[theorem]{}

\newtheorem{remark}[theorem]{Remark}

\theoremstyle{remark}
\newtheorem*{acknow}{Acknowledgements}

\numberwithin{equation}{theorem}

\newcommand{\G}{\Gamma}

\newcommand{\koz}{/\!\!/}

\renewcommand{\K}{\mathsf{K}}

\newcommand{\Kinj}{\mathsf{K}(\mathsf{Inj\,R})}

\newcommand{\Kc}{\K^{\mathsf{c}}}

\newcommand{\Q}{\mathsf{Q}}

\renewcommand{\i}{\mathsf{i}}

\newcommand{\col}{\colon}

\newcommand{\Z}{\mathbb{Z}}

\renewcommand{\H}{H}
\newcommand{\be}{\mathsf{s}\,}
\newcommand{\kxC}{C \koz \be}

\newcommand{\df}{\partial}
\newcommand{\qu}{quasi-isomorphism }

\begin{document}

\title[Finite injective dimension]{Finite injective dimension over rings \\with Noetherian cohomology}

\author{Jesse Burke}
\address{Department of Mathematics\\ 
Universit\"at Bielefeld\\ 
33501 Bielefeld\\ 
Germany.}
\email{jburke@math.uni-bielefeld.de}



\begin{abstract}
We study rings that have Noetherian cohomology over a
ring of cohomology operators. Examples of such rings include
commutative complete intersection rings and finite dimensional
cocommutative Hopf algebras.  The main result is a
criterion for a complex of modules over a ring with Noetherian
cohomology to have finite injective dimension. The criterion implies in particular that for any
module over such a ring, if all higher self-extensions of the
module vanish, then it must have finite injective dimension. This
generalizes a theorem of Avramov  and Buchweitz
for complete intersection rings, and
 a well-known theorem in the representation theory of finite groups
 from finitely generated to arbitrary modules.
\end{abstract}

\maketitle

\section{Introduction}
\newcommand{\ja}{\mathfrak{r}}
\newcommand{\hh}{\operatorname{H} \! \operatorname{H}^*(R\, | \, k)}
\newcommand{\hhev}{\operatorname{H} \! \operatorname{H}^{\text{ev}}(R\, |\, k)}

Let $R$ be an associative ring and $\S$ a ring of cohomology operators
on $R$. Thus $\S$ is a commutative graded ring and there exists a family of
homogeneous maps of graded rings indexed by complexes of
$R$-modules $M$: \[\zeta_M: \S \to \Ext * R M M,\]
that satisfies a certain commutativity condition. See Section \ref{sec:preliminaries}
for the full definition. We say $R$ has \emph{Noetherian cohomology}
over $S$ if $\Ext * R M M$ is a Noetherian $\S$-module via $\zeta_M$ for all $M$ with Noetherian cohomology over $R$.

In this paper we prove the following:
\begin{introthm}
\label{main_intro_thm}
 Let $R$ be a ring with Noetherian cohomology over a ring of cohomology operators $\S$,
and let $M$ be a complex of $R$-modules with $H^n(M) = 0$ for $n \gg
0$. Let $\S^+$ be the ideal $\oplus_{i \geq 1} \S^i $. If the $\S$-module $\Ext *  R M M$ is $\S^+$-torsion, then $M$ has finite injective dimension.
\end{introthm}
Recall that $\Ext *  R M M$ is $S^+$-torsion if for every
$x \in \Ext *  R M M$ there exists an integer $n$ such that
$(S^+)^n \, x = 0$. There is, for instance, an integer $l$ depending on the degrees of the
generators of $\S^+$, such that if $\Ext {nl} R M M = 0$ for some $n
\geq 1$, then $\Ext * R M M$ is $\S^+$-torsion; see \ref{cor}. A complex has finite injective dimension if it has a bounded above
semi-injective resolution, see \ref{def_semi_inj}. If the complex in question is a module, then a
semi-injective resolution is tan injective
resolution in the classical sense. To compute $\Ext * R
M M$ for a complex $M$, one may use a semi-injective resolution, and
so if $M$ is a module then $\Ext * R M M$ agrees with the classical
notion. Thus, a special case of the theorem is that if $M$ is an $R$-module with $\Ext n R M M = 0$ for $n \gg 0$, then $M$ has finite injective dimension.

There are many rings with Noetherian cohomology and hence
to which the result above applies. First, assume that $R$ is a ring
of the form $Q/(f_1, \ldots, f_c)$, where $Q$ is a commutative Noetherian regular ring
of finite Krull dimension and $f_1, \ldots, f_c$ is a $Q$-regular
sequence. The graded polynomial ring $\S = R[\chi_1, \ldots, \chi_c]$,
where the degree of each $\chi_i$ is 2, is a ring of cohomology
operators for $R$ and $R$ has Noetherian cohomology over $S$ by
\cite{Gu74}. In this context the Theorem
generalizes a key instance of \cite[Theorem 4.2]{AvBu00} from finitely
generated modules to a large class of complexes, including all modules:
\begin{Corollary}
\label{ci}
  Let $R = Q/(f_1, \ldots, f_c)$, where $Q$ is a commutative
  Noetherian regular ring of finite Krull dimension and $f_1, \ldots, f_c$ is a $Q$-regular sequence. Let $M$ be a complex of $R$-modules with $H^n(M) = 0$ for $n \gg 0$. If $\Ext {2n} R M M = 0$ for some $n \geq 1$, then $M$ has finite injective dimension.
\end{Corollary} 
Indeed, if $\Ext {2n} R M M = 0$ for some $n$, then $\Ext * R M M$
must be $\S^+$-torsion since the degree of $\chi_i$ is 2. Thus $l = 2$
in the notation above; see \ref{opers_gullik_thm} for further details.

Now let $R$ be a Hopf algebra over a field $k$. Any commutative subring of
$\Ext * R k k$ is a ring of cohomology operators on $R$; see
\ref{hopf_alg}. Let $S$ be the center of $\Ext * R k k$. It follows from the main result of
\cite{FS} that every finite
dimensional cocommutative Hopf algebra has Noetherian cohomology over
$S$. Thus we have the following:
\begin{Corollary}
Let $R$ be a finite dimensional cocommutative Hopf algebra over a
field $k$ and let $\S$ be the center of
$\Ext * R k k$. For an $R$-complex $M$ with $H^n(M) = 0$ for all $n \gg
0$, if $\Ext * R M M$ is $\S^+$-torsion, then $M$ has finite injective dimension.
\end{Corollary}
In particular the result applies to the group ring
of a finite group over a field where it generalizes a well known
result for finite dimensional representations to, in particular,
arbitrary representations.


For the proof of the main theorem, we work in an ``infinite completion'' of
the bounded derived category of Noetherian $R$-modules. This allows us
to avoid finiteness conditions on the complexes to which the criterion
is applied. By \cite{Kr05}, such a completion is given by the homotopy category of injective $R$-modules. We recall relevant facts about this category in Section 2. In Section 3 we give the precise definition of a ring of cohomology operators and prove a preliminary result. The proof of the theorem occupies Section 4 and in Section 5 we apply it to the cases discussed above.

The techniques in this paper are inspired by \cite{BIK08}. We have minimized the use of machinery from that paper to make this one more self-contained.

\begin{acknow}
This work formed part of my Ph.D. thesis. I would like to thank my
advisor Srikanth Iyengar for his support and guidance, Dave
Benson for discussing this material with me, and the
referee for providing several helpful suggestions including an improved
proof of Proposition \ref{Locid}. I was
partly supported by the NSF Grant DMS 0903493.
\end{acknow}

\section{Background}
\label{setup_sec}
Throughout $R$ denotes an associative ring. By the word ``module'' we mean a left-module. An $R$-complex is a complex of $R$-modules.

In this section we briefly recall some definitions and results on triangulated categories. We then review the homological algebra of complexes that we will need.
\begin{chunk}
Let $M$ be an $R$-complex. We write $\H^n(M)$ for the $n$th cohomology group of $M$ and $\H(M)$ for the graded $R$-module that in degree $n$ is $\H^n(M)$. We say $M$ has \emph{finite cohomology} if $\H(M)$ is a Noetherian $R$-module; this implies in particular that $\H^n(M) = 0$ for $|n| \gg 0$. The complex $M$ is \emph{acyclic} if $\H(M) = 0$.

Let $N$ be another $R$-complex. We denote the \emph{Hom-complex} between $M$ and $N$ by $\Hom R M N$; this has components and differential given by
\[ \Hom R M N ^ n = \prod_{i \in \Z} \Hom R {M^i} {N^{i+n}}  \qquad \df(f) = \df^N \circ f - (-1)^{|f|} f \circ \df^M, \]
where $|f|$ is the degree of $f$.
A \emph{morphism} $f: M \to N$ is a degree zero cycle of $\Hom R M N$,
i.e. $\df(f) = 0$. It is a \emph{\qu}when $\H(f): \H(M) \to \H(N)$ is an isomorphism.
\end{chunk}

\begin{chunk}
The \emph{homotopy category of injective $R$-modules}, denoted by
$\Kinj$, has as objects complexes of injective $R$-modules. The
morphisms between objects $X,Y$ are given by
\[ \Hom {\Kinj} X Y := \H^0( \Hom R X Y). \]
In other words, morphisms in $\Kinj$ are homotopy equivalence classes of morphisms of complexes.

The standard shift functor on $\Kinj$ is denoted $\Sigma$. Thus for a
complex \[X = \ldots \to X^{n-1} \to X^n \to X^{n+1} \to \ldots\]
we have that $(\Sigma X)^n = X^{n+1}$ and $\partial_{\Sigma X} = - \partial_X$. By $\Homs \K X Y$ we denote the $\mathbb{Z}$-graded abelian group that in degree $n$ is $\Hom \K X {\Sigma^n Y}.$ With multiplication given by composition $\Homs \K X X$ is a graded ring while $\Homs \K X Y$ is a bimodule with left action by $\Homs \K Y Y$ and right action by $\Homs \K X X$.
\end{chunk}

\begin{chunk}
\label{thick_loc_def}
The category $\Kinj$ is triangulated. For a proof and reference on
triangulated categories see e.g. \cite{Ve96}. A triangulated
subcategory of $\Kinj$ is \emph{thick} if it is closed under direct
summands; it is \emph{localizing} when it is closed under set-indexed
direct sums. Every localizing subcategory in $\Kinj$ is automatically thick, see
e.g. the proof of \cite[1.4.8]{MR1388895}.

For a subclass of objects $\C$ in $\Kinj$, we denote by $\Thick \K \C$, respectively $\Loc \K \C$, the smallest thick, respectively localizing, subcategory containing $\C$. One may realize these by taking the intersection of all thick, respectively localizing, subcategories containing $\C$.

An object $C \in \Kinj$ is \emph{compact} if the natural map
\[ \bigoplus_{i \in I} \Hom \Kinj C {X_i} \to \Hom \Kinj C {\bigoplus_{i \in I}  X_i} \]
is an isomorphism for any set of objects $\{ X_i \}_{i \in I}$ of $\Kinj$. We denote the collection of compact objects of $\Kinj$ by $\Kinj^{\mathsf c}$. 

When $R$ is left-Noetherian, \cite[2.3.1]{Kr05} shows that $\Kinj$ is \emph{compactly generated}, i.e. an object $X \in \Kinj$ is nonzero if and only if there exists a compact object $C \in \Kinj$ such that $\Hom \Kinj C X \neq 0$.
\end{chunk}


\begin{chunk}
\label{def_semi_inj}
A complex of injective modules $I$ is \emph{semi-injective} if for all acyclic complexes $A$, the complex $\Hom R A I$ is acyclic. When $I$ is semi-injective it has the following lifting property: for every morphism $\alpha: M \to I$ and every quasi-isomorphism $\beta: M \to N$ there exists a unique up to homotopy map $\gamma: N \to I$ making the following diagram commute:
\[ \xymatrix{ M \ar[r]^\beta_{\simeq} \ar[d]_{\alpha} & N
  \ar@{.>}[dl]^{\gamma} \\ I & }. \] 

A \emph{semi-injective resolution} of a complex $M$ is a quasi-isomorphism $\eta_M : M \to \i M$, where $\i M$ is semi-injective. Every complex has a semi-injective resolution; this was first proven in \cite{Sp88}. Moreover, by the lifting property, a semi-injective resolution is unique up to isomorphism in $\Kinj$.

When $M$ is a module, viewed as a complex concentrated in degree 0, a
semi-injective resolution of $M$ is just an injective resolution in the usual sense.
\end{chunk}

\begin{chunk}
\label{ext_inj_dim}
Let $\i M,\i N$ be semi-injective resolutions of complexes $M,N$, respectively. Define the derived $\operatorname{Hom}$ functors as \[\Ext n R M N := \Hom \K {\i M} {\Sigma^n \i N} \cong \H^n \Hom R {\i M} {\i N}. \] 
The lifting property of semi-injective complexes shows that $\Ext * R M N$ is independent of the choice of resolutions, up to isomorphism.

If there exists a semi-injective resolution $\eta_M : M \to \i M$ such that $(\i M)^n = 0$ for all $n \gg 0$, then we say $M$ has \emph{finite injective dimension} and write \emph{$\id_R M \li$}. 

\end{chunk}

\begin{chunk}
\label{desc_comp_obs}
Let $\D(R)$ be the unbounded derived category of $R$-modules, see
e.g. \cite{Ve96} for the definition. We denote by $Q$ the localization
functor $Q: \Kinj \to \D(R)$ that sends a complex to its image in the derived category. When $R$ is left-Noetherian \cite[2.3.2]{Kr05} shows that $Q$ restricts to an equivalence 
\[Q: \Kinj^{\mathsf c} \xra{\cong} \Df,\]
where $\Df$ is the full subcategory of $\D(R)$ of objects with finite
cohomology. The functor $Q$ has a right adjoint, denoted by $Q_\rho$,
which takes any complex to a semi-injective resolution, viewed as an
object of $\Kinj$ by \cite[3.9]{Kr05}.

When restricted to $\Df$, $Q_\rho$ gives an inverse to the equivalence
above. Thus
\emph{the compact objects of $\K$ are exactly the semi-injective
  resolutions of objects in $\Df$}. 
\end{chunk}

The following construction is a key part of the proof of the main theorem.
\begin{chunk}
\label{bos_loc}
Let $\tS = \Loc \K \C$, for a set of compact objects $\C$ in $\Kinj$. For any object $X$ in $\Kinj$ there is a triangle
\[ \G X \to X \to \Ls X \to \]
such that $\G X \in \tS$ and $\Ls X \in {\tS^{\perp}},$ where \[\tS^{\perp} = \brm{ Y \in \Kinj}{ \Hom \K Z Y = 0 \text{ for all } Z \in \tS}.\] This is a form of \emph{Bousfield localization}; see \cite[1.7]{Ne92} for a proof.
\end{chunk}

\section{Cohomology operators} 
\label{sec:preliminaries} 

\newcommand{\cat}{\Kinj}
\newcommand{\Rc}{R^{\mathsf{c}}}

Throughout this section $\S = \oplus_{i \geq 0} \S^i$ denotes a commutative graded ring .
\begin{chunk}
\label{coh_noeth}
We say $\S$ is a \emph{ring of cohomology operators} for $R$ if for every $X \in \cat$ there is a map of graded rings
\[ \zeta_X: \S \to \Homs {\cat} X X \] such that the two $\S$-module structures on $\Homs \cat X Y$ via $\zeta_X$ and $\zeta_Y$ agree. Thus for each $\alpha \in \Homs {\cat} X Y$, and all homogeneous $s \in \S$, we require
\begin{equation}
\label{comm_rels}
\zeta_Y( s ) \cdot \alpha = (-1)^{|s|}\alpha \cdot \zeta_X(s).
\end{equation}
We say $R$ has \emph{Noetherian cohomology} over $\S$ if $\S$ is a
Noetherian ring of finite Krull dimension and $\Homs \cat C C$ is a Noetherian $\S$-module for all compact objects $C$ in $\Kinj$.
\end{chunk}

\begin{remark}
\label{remark_def_nc}
Equivalently, $\S$ is a ring of cohomology operators for $R$ if there
is a ring map $\S \to \mathsf{Z}( \Kinj )$, where $\mathsf{Z}(-)$ denotes the graded
center of a triangulated category; see e.g.\ \cite[Section 4]{BIK08}.

A ring of cohomology operators for $R$ has been defined previously in
\cite{AvIy07} to be a ring map $\S \to \mathsf{Z}( \D(R) )$. The
essentially surjective
functor $\Q: \Kinj \to \D(R)$ induces a ring map $\mathsf{Z}(
\Kinj ) \to \mathsf{Z}( \D(R) )$ and thus a ring of cohomology
operators in our sense gives rise to a ring of cohomology operators in
the sense of \cite{AvIy07}.
\end{remark}

\newcommand{\cn}{cohomologically Noetherian }

In the rest of the section we assume that $\S$ is Noetherian, has finite Krull dimension, and is a ring of cohomology operators on $R$. We set $\S^+ = \oplus_{i \geq 1} \S^i$.

We will need the following result on the structure of a ring with Noetherian cohomology.
\begin{chunk}
\label{fin_pd} 
  Assume $R$ has Noetherian cohomology over $\S$. Then the
  following hold:
  \begin{enumerate}[\quad\rm(1)]
  \item $R$ is left-Noetherian;
  \item $\id_R R \li$;
  \item  An $R$-complex with finite cohomology $M$ has finite projective dimension if and only if $\Ext n R M
    M = 0 \text{ for all } n \gg 0$ if and only if $M$ has finite
    injective dimension.
  \end{enumerate}
This is contained in \cite{AvIy10}, where less assumptions are placed on $\S$. For the rings in Section \ref{apps} to which we apply the Theorem, the properties above are well-known.
\end{chunk}

\newcommand{\bs}{\mathsf{s}}
The following construction was introduced in \cite{BIK08}.
\begin{chunk}
\label{koz_ob}
Let $s$ be a homogeneous element of $\S$ of degree $n$ and let $X$ be an object of $\Kinj$. The \emph{Koszul object} of $s$ on $X$, denoted $X \koz s$, is the mapping cone of $\zeta_X(s) \in \Hom \Kinj X {\Sigma^n X}$. Thus there is an exact triangle
\begin{equation} X \xra {\zeta_X(s)} \Sigma^n X \to X \koz s\to, \label{koz_tri} \end{equation}
and $X \koz s$ is unique up to isomorphism.  For $\bs = s_1, \dots, s_r$ a sequence of homogeneous elements of $\S$, the Koszul object of $\bs$ on $X$, denoted $X \koz \bs$, is defined inductively as the Koszul object of $s_r$ on $X \koz ({s_1, \dots,s_{r-1}})$. 

Let $Y$ be another object of $\Kinj$. We need the following properties of Koszul objects:
\begin{enumerate}
\label{koz_is_compact}
\item If $X$ is compact, then so is $X \koz\bs$; this follows by
  induction and the triangle \eqref{koz_tri} above.
\item
\label{koz_ann} There exists an integer $n \geq 0$, independent of $X$ and $Y$, such that 
\[
  (\bs)^n \Homs {\Kinj} Y {X \koz \bs} = 0 = (\bs)^n \Homs {\Kinj} {X
    \koz \bs} Y,
\]
where $(\bs) = (s_1, \ldots, s_n)$ is the ideal in $\S$ generated by
$s_1, \ldots, s_n$.
\item
\label{koz_tor} If $\Homs \Kinj {X \koz \bs} Y = 0$ and the $\S$-module $\Homs \Kinj X Y$ is $\bs$-torsion then \[\Homs \Kinj X Y = 0.\]

\end{enumerate}
The last two results are contained in \cite[5.11]{BIK08}.
\end{chunk}

The next result shows that every compact object of $\Kinj$ can be cut down to an object with finite projective dimension using the above construction.
\begin{prop}
\label{koz_pd}
Assume $R$ has Noetherian cohomology over $\S$. Let $\be = s_1, \ldots, s_r$ be a set of generators of the ideal $\S^+ = \oplus_{ i > 0} \S^i$ and let $\i R \in \Kinj$ be an injective resolution of $R$. For every compact object $C$ of $\Kinj$ the object $\kxC$ is in $\Thick \K {\i R}$. In particular there is an inclusion of subcategories:  \[\Thick\K { C \koz  \be | \, C \in \Kinj^{\mathsf c} } \subseteq \Thick \Kinj {\i R} . \]
\end{prop}

\begin{proof}
By \ref{koz_ob}(2) there exists $n \geq 1$ such that $(\be)^n \Homs \Kinj {\kxC} {\kxC}  = 0.$ Since $\kxC$ is compact, the $\S$-module $\Homs  \Kinj {\kxC} {\kxC}$ is finitely generated by the definition of Noetherian cohomology. 
A standard argument now shows that
\begin{equation} \label{high_vanish}\Homd m {\Kinj} {\kxC} {\kxC} = 0 \text{ for } m \gg 0.\end{equation}
Since $\kxC$ is compact, by \ref{desc_comp_obs}, the complex $\kxC$ is semi-injective. Thus
\[ \Homs \Kinj {\kxC} {\kxC} \cong \Ext * R {\kxC} {\kxC}.\]
Now \ref{high_vanish} and \ref{fin_pd}(3) show that ${\kxC}$ has finite projective dimension. One checks, by induction on projective dimension for instance, that this implies that $\kxC \in \Thick {\D(R)} {R}.$ Since triangulated functors preserve thick subcategories we have that \[Q_\rho (\kxC) \in \Thick \Kinj {Q_\rho R}.\] As semi-injective resolutions are unique in $\Kinj$ and $\kxC$ and $Q_\rho (\kxC)$ are semi-injective, we have that $Q_\rho (\kxC) \cong \kxC$ and $Q_\rho R \cong \i R$. Stringing together the above shows that $\kxC$ is in $\Thick \K {\i R}$.
\end{proof}

\section{Finite Injective Dimension}
In this section we prove the theorem in the introduction. To do this we need the following:
\begin{prop}
\label{Locid}
Let $R$ be a left-Noetherian ring that has finite injective dimension
as a left $R$-module and let
$M$ be an $R$-complex with $\H^n(M) = 0$ for $n \gg 0$. Let
$\i R$ and $\i M$ be semi-injective resolutions of $R$ and $M$
respectively. If $\i M$ is in $\Loc\K {\i R}$, then $M$ has finite injective dimension.
\end{prop}

\begin{proof}
Since $M$ has right bounded cohomology, we may pick a projective
resolution $P \xra{\simeq} M$, i.e.\ a quasi-isomorphism such that $P^j$ is
projective and $P^j = 0$ for $j \gg 0$. Each $P^j$ has finite
injective dimension bounded by the injective dimension of the
ring, which we denote by $d$. 

Fix an injective resolution of each $P^j$ of length at most $d$. By the
comparison theorem there are maps between the resolutions which form a
bicomplex. Taking the total sum complex of this bicomplex gives a
complex $L$ and a quasi-isomorphism $P \xra{\simeq} L$ such that each $L^j$ is
injective and $L^j = 0$ for
$j \gg 0$. Now let $L \to \i L$ be a semi-injective resolution. We
have a diagram
\[ \xymatrix{ P \ar[r]^\simeq \ar[dr]^\simeq & L \ar[r]^{\simeq} & \i L \\
& M \ar[r]^\simeq & \i M } \]
By the lifting property of semi-injective resolutions, described in
\ref{def_semi_inj}, we see that $\i M \cong \i L$ in $\Kinj$. In
particular $\i L$ is a semi-injective resolution of $M$ and
$\i L \in \Loc \K {\i R}.$

Let $T$ be the mapping cone of $L \to \i L$. We have a triangle
\[ L \to \i L \xra{v} T \to \]
in $\Kinj$. Note that $T$ is acyclic since $L \to \i L$ is a
quasi-isomorphism. Thus, we have isomorphisms \[\Homs \K {\i R} T \cong \Homs {\mathsf{K}(R)} R T \cong \H^*(T) = 0.\] The first is \cite[2.1]{Kr05}, the second is clear, and the third is the fact that $T$ is acyclic.

The full subcategory whose objects are
\[ \brm{X}{\Homs \K X T = 0 }\]
is a localizing subcategory of $\Kinj$. Thus, since $\i R$ is in this subcategory,
so is $\Loc \K {\i R}$. In particular $\i L \in \Loc
\K {\i R}$, and thus $\Homs \K {\i L} T = 0.$ This shows that the map $v$ above is nullhomotopic. We will show that this forces $\i L$ to have an injective cokernel in a high degree.

Since $v$ is nullhomotopic there exists a map $s: \i L \to T$ such that $\partial s + s \partial = v$. Let $k$ be an integer such that $L^n = 0$ for all $n \geq k$, which exists by assumption. Thus $v^n$ is bijective for all $n \geq k$ and we have that $(v^n)^{-1} \partial s + (v^n)^{-1}  s \partial = 1_{\i L^{n}}$. One checks that $v^{-1}$ commutes with the differentials in the degrees for which it is defined; this gives
\[ \partial (v^{n-1})^{-1} s + (v^n)^{-1}  s \partial = 1_{\i L^{n}}. \]
Thus $v^{-1} s$ is a contracting homotopy of $1_{\i M}$ in high degrees. A simple diagram chase now shows that $\Img( \partial^{k})$ splits as a submodule of $(\i L)^{k+1}$ and hence is injective.

Since $v$ is a bijection in degrees $n \geq k$ and $T$ is acyclic, this
implies that $\H^n(\i L ) = 0$ for $n \geq k$. Thus $\i L$ has an
injective cokernel in a degree higher than its last nonzero
cohomology; by \cite[2.4.I]{AvFo91} this implies that $M$ has finite
injective dimension. One may also verify this directly by noting that
we've shown that $\i L \cong X \oplus Y$ with $X^i = 0$ for $i \gg 0$
and $Y$ nullhomotopic.
\end{proof}

\begin{theorem}
\label{main_thm}
Let $R$ be an associative ring and $\S$ a Noetherian graded ring of
finite Krull dimension. Assume that $\S$ is a ring of cohomology
operators on $R$ and that $R$ has Noetherian cohomology over $S$. For an $R$-complex $M$ with $\H^n(M) = 0$ for $n \gg 0$, if the $\S$-module $\Ext * R M M$ is $\S^+ = \oplus_{i \geq 1} \S^i$-torsion, then $M$ has finite injective dimension. 
\end{theorem}

\newcommand{\ix}{X}
\newcommand{\Ss}{\C}
\newcommand{\Rs}{\mathsf{T}}

\begin{proof}
Let $X = \i M$ be a semi-injective resolution of $M$. Then, by \ref{ext_inj_dim},
\[\Ext *  R M M \cong \Homs \Kinj X X.\] Let $\be$ be a finite set of generators
of the ideal $\S^+$. By \ref{fin_pd}, $R$ is left-Noetherian and has finite injective dimension. Thus by \ref{Locid} it is enough to show that $\i M \in \Loc \K {\i R}$.
Since every localizing
subcategory in $\Kinj$ is thick, see \ref{thick_loc_def}, Proposition
\ref{koz_pd} shows that
\begin{equation}\Loc \K { C \koz  \be | \, C \in \Kc } \subseteq \Loc \K {\i R} . \label{loc_inclusion}\end{equation}
Thus to prove the theorem it is enough to show that $X \in \Loc \K { C
  \koz \be \, | \, C \in \Kinj^{\mathsf{c}} }.$ Let us set $\Ss := \Loc \K { C \koz \be \, | \, C \in \Kinj^{\mathsf{c}} }.$

Fix a compact object $D$. By hypothesis $\Homs \Kinj {\ix} {\ix}$ is $\S^+$-torsion. By the definition of cohomology operators, the action of $\S$ on $\Homs \Kinj D \ix$ factors through $\Homs {\Kinj} \ix \ix$ and hence $\Homs \Kinj D \ix$ is also $\S^+$-torsion. 

Now consider the full subcategory $\Rs$ of $\Kinj$ with objects those $Z \in \Kinj$ such that $\Homs \Kinj D Z$ is $\S^+$-torsion.
It is clearly closed under suspension. Given a triangle $Y \to Z \to W \to \Sigma Y$ in $\Kinj$ there is an exact sequence of $\S$-modules:
\[ \Homs \Kinj D Y \to \Homs \Kinj D Z \to \Homs \Kinj D W.\]
From this we see that if $\Homs \Kinj D Y$ and $\Homs \Kinj D W$ are $\S^+$-torsion then $\Homs \Kinj D Z$ is as well. This shows that $\Rs$ is triangulated.
For a family of objects $\{ Z_i \}$ in $\T$, we have that
\[ \Homs \Kinj D {\bigoplus_i Z_i} \cong \bigoplus_i \Homs \Kinj D {Z_i},\]
since $D$ is compact. Thus $\Rs$ is closed under direct sums and hence is localizing. By \ref{koz_ob}(2), for every object $C$ the module $\Homs \Kinj D {C \koz \be}$ is $S^+$-torsion. Thus 
\[ \Ss = \Loc \K { C \koz \be \, | \, C \in \Kinj^{\mathsf{c}} } \subseteq \Rs\] since $\Rs$ is localizing and each $C \koz \be$ is in $\Rs$. 

Since $\Ss$ is compactly generated there is a triangle
\begin{equation}
\G \ix \to \ix \to \Ls \ix \to \label{loc_seq} 
\end{equation}
with $\G \ix \in \Ss$ and $\Ls \ix \in \Ss^{\perp}$; see
\ref{bos_loc}. We have that $\Homs \Kinj D {\G \ix}$ is $\S^+$-torsion since $\G \ix
\in \Ss \subseteq \Rs$. We have shown above that $X \in \Rs$. Thus $\Ls X \in \Rs$ since $\Rs$ is triangulated. By definition this means $\Homs \Kinj D {\Ls \ix}$ is $S^+$-torsion. Since $D \koz \be \in \Ss$ and $\Ls \ix \in \Ss^{\perp}$, we have that
\[ \Homs \Kinj {D \koz \be} {\Ls \ix} = 0. \]
 By \ref{koz_ob}(3) this implies that $\Homs \Kinj D {\Ls \ix} =
 0$. But since $D$ was an arbitrary compact object and $\Kinj$ is
 compactly generated, see \ref{thick_loc_def}, this shows that  $\Ls
 \ix = 0$. By the triangle \eqref{loc_seq} this implies that $\G \ix
 \cong \ix \in \Kinj$ and hence $\ix$ is an object of  $\Ss = \Loc \K { C \koz \be \, | \, C \in \Kinj^{\mathsf{c}} }$.
\end{proof}

\begin{remark}
The hypothesis that $\H^n(M) = 0$ for $n \gg 0$ is
necessary. Indeed, from the definition of finite injective dimension,
given in \ref{def_semi_inj}, if a complex
$M$ has finite injective dimension, then $\H^n(M) = 0$ for $n \gg 0$.
\end{remark}

We record the following which was contained in the proof of \ref{main_thm}.
\begin{cor}
\label{equality}
Under the assumptions and notation of Theorem \ref{main_thm}, there is an equality
\[ \Loc \K  {C \koz \be \, | \, C \in \Kinj^{\mathsf{c}} } =\Loc \K {\i R}.\]
\end{cor}

\begin{proof}
One containment is given by \eqref{loc_inclusion}. For the other
direction, note that since the $\S$-module $\Homs \Kinj {\i R} {\i R}
\cong \Ext * R R R$ is $\S^+$-torsion, the proof of \ref{main_thm} above shows that $\i R \in \Loc \K  {C \koz \be \, | \, C \in \Kinj^{\mathsf{c}} }$.
\end{proof}

\begin{cor}
\label{cor}
Let $R, \S$ and $M$ be as in \ref{main_thm}. Let $s_1, \ldots, s_r$ be a finite set of homogeneous generators of the ideal $\S^+$. Set \[ d := \max \brm{ \deg s_i }{ 1 \leq i \leq r} \text{  and  } l := \operatorname{lcm} \brm{\deg s_i}{ 1 \leq i \leq r}.\]
Then $\id_R M \li$ if one of the following holds:
\begin{enumerate}
\item there exists an integer $n \geq 0$ such that $\Ext {j} R M M = 0$ for all $n \leq j \leq n + d - 1$; or
\item there exists an integer $m \geq 0$ such that $\Ext {ml} R M M = 0$.
\end{enumerate}
\end{cor}

\begin{proof}
Either condition forces the $\S$-module $\Ext * R M M$ to be
$\S^+$-torsion. Indeed, assume that there exists an integer $n$ such
that (1) holds. For every $i$ there exists an integer $k_i$ such that
\[n \leq k_i(\deg s_i) \leq n + d -1.\]
One way to see this is by induction on $n$. Consider the ideal $(\S^+)^{k_1+\ldots+k_r } = (s_1, \ldots, s_r)^{k_1+\ldots+k_r}$ in $\S$. It is generated by monomials in the $s_i$ of the form $s_1^{n_1}\ldots s_r^{n_r}$ for positive integers $n_i$ with $\sum n_i = \sum k_i$. For each such monomial there exists an $i$ such that $n_i \geq k_i$, else $\sum n_i < \sum k_i$; applying $\zeta_M$ to the monomial, and using that $\zeta_M$ is a map of rings, we see that 
\begin{align*}\zeta_M( s_1^{n_1}\ldots s_r^{n_r}) &= \zeta_M(s_1^{n_1})...\zeta_M(s_i^{n_i})...\zeta_M(s_r^{n_r}) \\
&= \zeta_M(s_1^{n_1})...\zeta_M(s_i^{k_i})\zeta_M(s_i^{n_i - k_i})...\zeta_M(s_r^{n_r}) = 0\end{align*} since $\zeta_M(s_i^{k_i}) \in \Ext {k_i(\deg s_i)} R M M = 0.$ Thus \[(\S^+)^{k_1+\ldots+k_r }\Ext * R M M = \zeta_M((\S^+)^{k_1+\ldots+k_r }) \Ext * R M M = 0\] and hence $\Ext * R M M$ is $S^+$-torsion. By Theorem \ref{main_thm} this shows that $\id_R M \li$.

To prove $(2)$ assume that such an $m$ exists. For every $i = 1,
\ldots, r$, there exists an integer $d_i$ such that $d_i (\deg s_i) =
l$. Letting $\alpha = m( \sum_i d_i)$, a similar proof as above shows
that $(s_1, \ldots, s_r)^\alpha \Ext *  R M M  = 0$.
\end{proof}

\section{Applications}
\label{apps}
In this section we apply Theorem \ref{main_thm} in the two contexts discussed in the introduction.

\begin{chunk}
Let $R$ be a commutative ring with a presentation \[R \cong Q/(\bf),\]
where $Q$ is a commutative Noetherian regular ring of finite Krull dimension and $(\bf) = (f_1, \ldots, f_c)$ is a $Q$-regular sequence.

\label{opers_gullik_thm}
Let $\S = R[\chi_1, \dots, \chi_c]$ be the polynomial ring in $c$ indeterminates over $R$, graded by setting $|\chi_i| = 2$.
For every $X \in \Kinj$ there is a homomorphism of graded $R$-algebras
\[ \zeta_X : \S \to \Homs \Kinj X X. \]
When $X = \i M$ is the injective resolution of a finitely generated $R$-module $M$, so that \[\Homs \Kinj X X \cong \Ext * R  M M,\] such a map $\zeta_X$ may be constructed as in \cite[Section 1]{Ei80} using a free resolution of $M$. The process described in \cite[Section 1]{Av89}, which replaces free resolutions with injective resolutions, generalizes to arbitrary objects of $\Kinj$. The results of \emph{loc. cit.} show that the maps $\zeta_X$ satisfy the conditions of a ring of cohomology operators.

By \cite[5.1]{AvSu98} the $\S$-module $\Homs \Kinj {\i M} {\i M} \cong
\Ext * R M M$ is finitely generated when $M$ has finite cohomology
over $R$. This was proved first by Gulliksen \cite{Gu74} for
modules. It follows \ref{desc_comp_obs} that $R$ has Noetherian cohomology over $\S$. Restating Theorem \ref{main_thm} in this context, we have:
\begin{cor}
Let $Q$ be a commutative Noetherian regular ring of finite Krull
dimension, $(\bf) = (f_1, \ldots, f_c)$ a $Q$-regular sequence and $R
= Q/(\bf)$. For an $R$-complex $M$ with $H^n(M) = 0$ for all $n \gg
0$, if $\Ext * R M M$ is $\S^+$-torsion, then $M$ has finite injective dimension. \qed
\end{cor}

In the notation of Corollary \ref{cor} we see that $d = 2 = l$. Since $R$ is a Gorenstein ring of finite Krull dimension, a module has finite projective dimension if and only if it has finite injective dimension. This gives:
\begin{cor}
If $M$ is an arbitrary $R$-module such that $\Ext {2n} R M M = 0$ for some $n \geq 1$ then $M$ has finite projective dimension. \qed
\end{cor}
\begin{remark}
  In \cite[4.2]{AvBu00} the same statement is proved for finitely
  generated modules of finite complete intersection dimension over a
  Noetherian ring. All finitely generated modules over the ring $R$ have finite complete intersection dimension. However, complete intersection dimension is not defined for non-finitely generated modules, so we have not generalized completely \cite[4.2]{AvBu00}.
\end{remark}

\end{chunk}

\begin{chunk}
\label{hopf_alg}
Let $R$ be a Hopf algebra over a field $k$. For two $R$-modules $M, N$
we view $M \otimes_k N$ as an $R$-module via the diagonal map $\Delta:
R \to R \otimes_k R$. When $M, N$ are injective then so is $M
\otimes_k N$. For $X \in \Kinj$ the functor $-\otimes_k X$ preserves
homotopies of maps. Thus there is a functor $-\otimes_k X \col \Kinj
\to \Kinj$. Viewing $k$ as an $R$-module via the augmentation there is an isomorphism
\[\varphi_X: \i k \otimes_k X \xra{\cong} X, \]
see \cite[5.3]{BK} which proof holds in our more general situation.
Thus for each $X$ one gets a map
\[
\eta_X \col \Homs \K {\i k} {\i k} \to \Homs \K X X
\]
that sends $\alpha: \i k \to \Sigma^n \i k$ to \[\varphi_{\Sigma^n X} ( \alpha \otimes_k X) (\varphi_X)^{-1}: X \to \Sigma^n X.\]
\newcommand{\even}{\operatorname{even}}
\newcommand{\sev}{\S^{\even}}
One can check that $\eta_X$ is a ring map. Let $\S$ be the ring $\Ext * R k k \cong \Homs \K {\i k} {\i k}$. By \cite[(VIII.4.7), (VIII.4.3)]{Mac63} the ring $S$ is graded-commutative and the maps $\eta_X$ satisfy the commutativity relations \eqref{comm_rels}. Thus setting
\[ \sev := \left \{ \begin{array}{cr} \bigoplus_{i \geq 0} \Ext {2i} R k k & \text{ if } \operatorname{char} k \neq 2 \\
& \\
    \Ext * R k k & \text{ if } \operatorname{char} k = 2 \end{array} \right . \]
we see that $\sev$ is commutative and is a ring of cohomology operators on $R$.

By the main result of \cite{FS}, when $R$ is cocommutative and finite dimensional over $k$, the ring $\S$ is Noetherian and $\Ext * R M N$ is a Noetherian $\S$-module (via $\eta_M$, or equivalently, $\eta_N$) for all complexes $M, N$ with finite cohomology. The ideal of odd degree elements in $\S$ is nilpotent when $\operatorname{char} k \neq 2$. Thus when $R$ is a cocommutative finite dimensional Hopf algebra it has Noetherian cohomology over $\sev$.

Specializing Theorem \ref{main_thm} and Corollary \ref{cor} to this context, and using that $R$ is self-injective, we have:
\end{chunk}
\begin{cor}
Let $R$ be a finite dimensional cocommutative Hopf algebra and $\S^{\operatorname{even}}$ the commutative ring defined as above. For an $R$-complex $M$ with $H^n(M) = 0$ for all $n \gg
0$, if $\Ext * R M M$ is $\S^+$-torsion, then $M$ has finite injective
dimension. \qed
\end{cor}

\begin{cor}
Let $R$ be as above and $M$ an $R$-module. There exists an integer $l$
such that if $\Ext {nl} R M M = 0$ for some $n \geq 1$ then $M$ has
finite projective dimension. \qed
\end{cor}




\renewcommand{\baselinestretch}{1.1}
\renewcommand{\MR}[1]{%
  {\href{http://www.ams.org/mathscinet-getitem?mr=#1}{MR #1}}}
\providecommand{\bysame}{\leavevmode\hbox to3em{\hrulefill}\thinspace}
\newcommand{\arXiv}[1]{%
  \relax\ifhmode\unskip\space\fi\href{http://arxiv.org/abs/#1}{arXiv:#1}}

\newcommand{\etalchar}[1]{$^{#1}$}
\providecommand{\bysame}{\leavevmode\hbox to3em{\hrulefill}\thinspace}
\providecommand{\MR}{\relax\ifhmode\unskip\space\fi MR }
\providecommand{\MRhref}[2]{%
  \href{http://www.ams.org/mathscinet-getitem?mr=#1}{#2}
}
\providecommand{\href}[2]{#2}

\end{document}